\documentclass{article}
\usepackage[colorlinks=true,urlcolor=blue, linkcolor=blue,pageanchor=false ]{hyperref}
\usepackage{geometry}
\usepackage[english]{babel}
\usepackage{amsmath}
\usepackage{graphicx}
\usepackage{ marvosym }
\usepackage[utf8]{inputenc}
\usepackage{mathtools}
\usepackage{geometry}
\usepackage{amsfonts}
\usepackage{amssymb}
\usepackage{amsthm}
\usepackage{thmtools}
\usepackage{t1enc}
\usepackage[titles]{tocloft}
\usepackage{wasysym}
\usepackage{stmaryrd}
\usepackage{calc}  
\usepackage{enumitem} 
\usepackage{bookmark}
\usepackage{tikz}
\usetikzlibrary{arrows, shapes,snakes}
\usetikzlibrary{decorations}
\usepackage{float}
\usepackage{graphicx}
\usepackage{indentfirst}

\theoremstyle{definition}

\theoremstyle{plain}
\newtheorem{thm}{Theorem}

\newtheorem{prop}[thm]{Proposition}
\newtheorem*{prop*}{Proposition}
\newtheorem*{seged*}{Sublemma}
\newtheorem{cor}[thm]{Corollary}
\newtheorem{lem}[thm]{Lemma}

\newtheorem{cond}[thm]{Condition}
\newtheorem*{cond*}{Condition}

\newtheorem*{lem*}{Lemma}
\theoremstyle{definition}

\newtheorem*{defn*}{Definition}

\newtheorem{fel*}[thm]{Exercise}

\newtheorem*{megf*}{Observation}
\theoremstyle{remark}
\newtheorem{rem}[thm]{Remark}

\newtheorem{obs}[thm]{Observation}
\newtheorem*{rem*}{Remark}

\newenvironment{sbiz}{\par\noindent{\itshape Proof:}\ }{\newmoon}
\newenvironment{ssbiz}{\par\noindent{\itshape Proof:}\ }{\newmoon\newmoon}

\title{Countable Menger theorem with finitary matroid constraints on the ingoing edges }

\author{
\fi Attila Joó\thanks{MTA-ELTE 
Egerváry 
Research Group, 
Department of Operations Research, Eötvös Loránd University, 
Budapest, Hungary. 
E-mail: {\tt 
joapaat@cs.elte.hu}.}}

\date{\the\year}
\begin{document}

\maketitle

\begin{abstract}
 We present a strengthening of the countable Menger theorem  \cite{aharoni1987menger} (edge version) of R. Aharoni (see also in 
 \cite{diestel2005graph} p. 217). Let $ D=(V,A) $ be a countable digraph with $ s\neq t\in V $ and let $ \mathcal{M}=\bigoplus_{v\in 
 V}\mathcal{M}_v $ be a matroid on $ A $ where $ \mathcal{M}_v $ is a finitary matroid on the ingoing edges of $ v $. We show that there 
 is a system of edge-disjoint $ s \rightarrow t $ paths $ \mathcal{P} $ such that the united edge set of the paths is $ \mathcal{M} 
 $-independent, and 
 there is a $ C \subseteq A $ consists of one edge from each element of $ \mathcal{P} $ for which $ \mathsf{span}_{\mathcal{M}}(C) $ 
 covers all the $ s\rightarrow t $ paths in $ D $.
\end{abstract}

\section{Notation}
The variables $ \xi, \zeta $ denote ordinals and $ \kappa $ stands for an infinite cardinal. We write $ \omega $  for the smallest limit 
ordinal (i.e. the set of natural numbers). We apply the abbreviation $ H+h $ for the set $ H\cup \{ h \} $ and $ H-h $ for $ H\setminus 
\{ h \} $ and we denote by $ \triangle $ the symmetric difference (i.e. $ H\triangle J:=(H\setminus J )\cup (J \setminus H) $).

The digraphs $ D=(V,A) $ of this article could be arbitrarily large and may have multiple 
edges and loops (though the later is 
irrelevant). For $ X \subseteq V $ we denote the ingoing and the outgoing edges of $ X $ in $ D $ by $ \mathsf{in}_D(X) $ and $ 
\mathsf{out}_D(X) $. We 
write $ D[X] $ for the subdigraph induced by the vertex set $ X $. If $ e $ is an edge  from vertex $ u $ to vertex $ v 
$, then we write $ 
\mathsf{tail}(e)=u $ and $ \mathsf{head}(e)=v $. The 
paths in this paper are assumed to be finite and directed. Repetition of vertices is forbidden in them (we say walk if we want to allow 
it).  For a path  $ P $ we 
denote by $ \mathsf{start}(P) $ and $ \mathsf{end}(P) $ the first and the last vertex of $ P $. If $ X,Y \subseteq V $, then $ P $ is a $ 
X \rightarrow Y $ path if $ V(P)\cap X= \{ \mathsf{start}(P) \} $ and $ V(P)\cap Y= \{ \mathsf{end}(P) \} $. For singletons  we simplify 
the notation and write $ x \rightarrow y $ instead of $ \{ x \}\rightarrow \{ y \} $. For a path-system (set of paths) $ \mathcal{P} $ we 
denote $ \bigcup_{P\in \mathcal{P}}A(P) $ by $ A(\mathcal{P}) $ and we write for the set of the last edges of the elements of $ 
\mathcal{P} $ simply $ 
A_{last}(\mathcal{P}) $. An $ s $-arborescence  is a directed tree in which every vertex is reachable (by a directed path) from its 
vertex $ s $.

If $ \mathcal{M} $ is a matroid and $ S $ is a subset of its ground set, then $ \mathcal{M}/S $ is the matroid we obtain by the 
contraction of the set $ S $.  We use $ \bigoplus $ for the direct sum of matroids. For the rank 
function we write $ r $ and $ \mathsf{span}(S) $ is the union of $ S $ and the loops (dependent singletons) of $ \mathcal{M}/S $. Let us 
remind 
that a matroid is called finitary if all of 
its circuits are finite.  
One can find a good survey about infinite matroids from the basics in \cite{nathanhabil}.

\section{Introduction}

In this paper we generalize the countable version of Menger's theorem of Aharoni \cite{aharoni1987menger} by applying  the results of 
Lawler and Martel about polymatroidal flows (see \cite{lawler1982computing}). 

Let us recall Menger's theorem (directed, edge version). 

\begin{thm}[Menger]\label{Menger tetel}
Let $ D=(V,A) $ be a finite digraph with $ s\neq t \in V $. Then the maximum number of the pairwise edge-disjoint $ s\rightarrow t $ 
paths is equal to the minimal number of edges that cover all the $ s \rightarrow t $ paths.
\end{thm}

Erdős observed during his school years that the theorem above remains true for infinite digraphs (by saying cardinalities instead of 
numbers). He felt that this is not the ``right'' infinite generalization of the finite theorem and he conjectured the 
``right''  generalization  which was known as the Erdős-Menger conjecture. It is based on the observation that in Theorem \ref{Menger 
tetel}  an  
optimal cover consists of one edge from each path of an optimal path-system. The Erdős-Menger conjecture states that 
for arbitrary large digraphs there 
is a path-system and a cover that satisfy these complementarity conditions. After a long sequence of partial 
results the countable case
has been settled 
affirmatively   by R. Aharoni:
\begin{thm}[R. Aharoni, \cite{aharoni1987menger}]
Let $ D=(V,A) $ be a countable digraph with $ s\neq t \in V $. Then there is a system $ \mathcal{P} $ of edge-disjoint $ s\rightarrow t $ 
paths such that there is a edge set $ C $ that covers all the $ s \rightarrow t $ paths in $ D $ and $ C $ consists of choosing one edge 
from each $ P\in \mathcal{P} $. 
\end{thm}

It is worth to mention that R. Aharoni and E. Berger proved the Erdős-Menger  conjecture in its full generality  in 2009 
(see \cite{aharoni2009menger}) which was one of 
the greatest achievements in the theory of infinite graphs. We present the following strengthening of the countable Menger's theorem 
above.  
 
 \begin{thm}\label{Menger fő0}
Let $ D=(V,A) $ be a countable digraph with $ s\neq t \in V $. Assume that there is a finitary matroid $ \mathcal{M}_v 
 $ on the 
 ingoing edges of  $ v $ for any $ v\in V $. Let $ \mathcal{M} $ be the direct sum of the  matroids $ \mathcal{M}_v $. Then there 
 is a system of edge-disjoint $ s \rightarrow t $ paths $ \mathcal{P} $ such that the united edge set of the paths is $ \mathcal{M} 
  $-independent, and 
  there is an edge set $ C $ consists of one edge from each element of $ \mathcal{P} $ for which $ \mathsf{span}_{\mathcal{M}}(C) $ 
  covers all the $ s\rightarrow t $ paths in $ D $.
 \end{thm}

Instead of dealing with covers directly, we are focusing on  $ t-s $ cuts ($ X\subseteq V $ is 
a $ t-s $ cut if $t\in  X \subseteq V\setminus \{ s \} $). Let us call a path-system $ \mathcal{P} $ independent if $ A(\mathcal{P}) $ is 
independent in $ \mathcal{M} $. Suppose that an independent system $ \mathcal{P} $ of edge-disjoint $ s\rightarrow t $ paths and a 
$ t-s $ cut $ X $ satisfy the \textbf{complementarity conditions}:

\begin{cond}\label{Meng complementary}
\leavevmode

   \begin{enumerate}
   \item $ A(\mathcal{P})\cap \mathsf{out}_D(X)= \varnothing $,
   \item  $ A(\mathcal{P})\cap \mathsf{in}_D(X) $ spans $ \mathsf{in}_D(X) $ in $ \mathcal{\mathcal{M}} $. 
   \end{enumerate}
   
\end{cond} 

Then clearly $ \mathcal{P} $ and $C:= A(\mathcal{P})\cap \mathsf{in}_D(X) $ satisfy the demands of Theorem \ref{Menger fő0}.
Therefore it is enough to prove the following reformulation of the theorem.

\section{Main result}

\begin{thm}\label{főtétel  atroid mengerb}
Let $ D=(V,A) $ be a  countable digraph with $ s\neq t\in V $ and suppose that there is a finitary matroid $ \mathcal{M}_v $ on $ 
\mathsf{in}_D(v) $ for each 
$ v\in V $ and let $ \mathcal{M}=\bigoplus_{v\in V}\mathcal{M}_v $. Then there is a system $ \mathcal{P} $ 
of edge-disjoint $ s\rightarrow t $ paths where $ A(\mathcal{P}) $ is independent in $ \mathcal{M} $ and a $ t-s $ cut $ X  $  
such that $ 
\mathcal{P} $ and $ X $ satisfy the 
complementarity conditions (Condition \ref{Meng complementary}). 
\end{thm}

\begin{proof}
Without loss of generality we may assume that $ \mathcal{M} $ does not contain loops.
A pair $ (\mathcal{W},X) $ is called a  \textbf{wave}  if $ X $ is 
a $ t-s $ cut  and $ \mathcal{W} $ is an independent system of edge-disjoint  $ s \rightarrow X $ paths such that the second 
complementarity 
condition holds for $ \mathcal{W} $ and $ X $  (i.e.  $ A_{last}(\mathcal{W}) $ spans $ \mathsf{in}_D(X) $ in $ \mathcal{M} $).

\begin{rem}\label{exists wave}
 By picking an arbitrary base $ B $ of $ \mathsf{out}(s) $ and taking $ \mathcal{W}:= B $ as a set of 
single-edge paths and $ X:= V\setminus \{ s \} $ we obtain a wave $ (\mathcal{W},X) $ thus always exists some wave.
\end{rem}

  We say that 
the wave $ (\mathcal{W}_1,X_1) $ extends the wave $ (\mathcal{W}_0,X_0) $  and write $ (\mathcal{W}_0,X_0)\leq (\mathcal{W}_1,X_1) $ if

\begin{enumerate}
\item $ X_1 \subseteq X_0 $,
\item $ \mathcal{W}_1 $ consists of the forward continuations 
  of some of the paths in  $ \mathcal{W}_0 $ such that the continuations lie in  $ X_0 $,
  \item $ \mathcal{W}_1 $ contains all of those  paths of $ \mathcal{W}_0 $ that meet $ X_1 $.
\end{enumerate}
 
   If in 
  addition $ \mathcal{W}_1 $ 
  contains a forward-continuation 
of \emph{all} the 
elements of $ \mathcal{W}_0 $, then the extension is called \textbf{complete}.  Note that  $ \leq $ is a partial ordering on the waves 
and if $ (\mathcal{W}_0,X_0)\leq (\mathcal{W}_1,X_1) $ holds, then the extension is proper (i.e. $ (\mathcal{W}_0,X_0)< 
(\mathcal{W}_1,X_1) $ ) iff $ X_1\subsetneq X_0$. 
\begin{obs}\label{incomplete extension}
If $ (\mathcal{W}_1,X_1) $ is an incomplete extension of $ (\mathcal{W}_0,X_0) $, then it is a proper extension  thus $ X_1 \subsetneq 
X_0 $. Furthermore,  $ \mathcal{W}_1 $ and $ X_0 $ 
do not  satisfy the second complementarity  condition (Condition \ref{Meng complementary}/2).
\end{obs}

\begin{lem}\label{sup of waves}
If a nonempty set $ \mathcal{X} $ of waves is linearly ordered by  $ \leq $, then $ \mathcal{X} $ has a unique smallest upper bound $ 
\sup(\mathcal{X}) $.  
\end{lem}

\begin{sbiz}
We may suppose that $ \mathcal{X} $ has no maximal element. Let $ \left\langle (\mathcal{W}_{\xi}, X_{\xi}): \xi<\kappa  \right\rangle 
$ be a cofinal sequence of $ (\mathcal{X},\leq) $.  We 
define  $ X:=\bigcap_{\xi<\kappa} X_\xi $ and  \[ \mathcal{W}:=\bigcup_{\zeta<\kappa}\bigcap_{\zeta<\xi}\mathcal{W}_\xi.%\{ P\in 
%\bigcup_{\xi<\kappa}\mathcal{W}_{\xi} : \mathsf{end}(P)\in X\}%  
\] For $ P\in \mathcal{W} $ we have $  V(P)\cap X_\xi=\{ \mathsf{end}(P) \} $ for all large enough $ \xi<\kappa $ hence $ V(P)\cap X=\{ 
\mathsf{end}(P)\} 
$. The paths in $ \mathcal{W} $ are pairwise edge-disjoint since $ P_1,P_2\in \mathcal{W} $ implies that $ P_1,P_2\in \mathcal{W}_{\xi} $ 
for all large enough $ \xi $.  Since the matroid $ \mathcal{M} $ is finitary the same argument shows that  $ \mathcal{W} $ is 
independent.

 Suppose that $ e\in  \mathsf{in}_D(X)\setminus A(\mathcal{W}) $. For a large enough $ \xi<\kappa $ we have $ e\in 
 \mathsf{in}_D(X_{\xi}) $. Then the last edges of those elements of $ \mathcal{W}_\xi $ that terminate in $ \mathsf{head}(e) $ spans $ e 
 $ in $ \mathcal{M} $. These paths have to be elements of all the further waves of the sequence (because of the 
 definition of $ \leq $) and thus of $ \mathcal{W} $ as well. Therefore $ (\mathcal{W},X) $ is a wave and clearly an upper bound.
 
 Suppose that $ (\mathcal{Q},Y)  $ is another upper bound for $ \mathcal{X} $. Then $ X_\xi \supseteq Y$ for all $ \xi<\kappa $ and  
 hence $ X 
 \supseteq Y $. Let $ Q\in \mathcal{Q} $ be arbitrary.  We know that $ \mathcal{W}_\xi $ contains an initial segment $ Q_\xi $ of $ Q $ 
 for all $ 
 \xi<\kappa $  because  $ (\mathcal{Q},Y)  $ is an upper bound (see the definition of $ \leq $). For $ \xi<\zeta<\kappa $ the path $ 
 Q_\zeta $ is a (not necessarily proper) forward continuation of $ Q_\xi $. From some index the sequence $ \left\langle Q_\xi: 
 \xi<\kappa  \right\rangle $ need to be constant, say $ Q^{*} $, since $ Q $ is a finite path. But then $ Q^{*}\in \mathcal{W} $. 
 Thus 
 any $ Q\in \mathcal{Q} $ is a forward continuation of a path in $ \mathcal{W} $. Finally assume that some $ P\in \mathcal{W} $ meets $ Y 
 $. 
 Pick a $ \xi<\kappa $ for which $ P\in \mathcal{W}_\xi $. Then $ (\mathcal{W}_\xi,X_\xi)\leq (\mathcal{Q},Y) $ guarantees $ P\in 
 \mathcal{Q} $. Therefore $(\mathcal{W},X)\leq(\mathcal{Q},Y).  $
\end{sbiz}\\

The Remark \ref{exists wave} and Lemma \ref{sup of waves} imply via Zorn's Lemma the following.
\begin{cor}\label{maximal wave}
There exists a maximal wave. Furthermore, 
there is a maximal wave which is greater or equal to an arbitrary prescribed wave.
\end{cor}

Let $ (\mathcal{W},X) $ be a maximal wave. To prove Theorem \ref{főtétel  atroid mengerb} it is enough to show that there is an 
independent 
system of edge-disjoint $ 
s\rightarrow t $ paths $ \mathcal{P} $ that consists of the forward-continuation of all the paths in $ \mathcal{W} $. Indeed, condition $ 
A(\mathcal{P})\cap \mathsf{out}_D(X)=\varnothing $ will be true automatically (otherwise $ \mathcal{P} $ would violate  independence, 
when 
 the violating path ``comes back'' to $ X $) 
and hence  $ 
\mathcal{P} $ and $ X $ will satisfy the complementarity conditions.

We need a method developed by Lawler and Martel in \cite{lawler1982computing} for the augmentation of polymatroidal flows in 
finite networks which works in the infinite case as well.

\begin{lem}\label{augmenting walk basic}
Let  $ \mathcal{P} $ be an independent system of edge-disjoint $ s\rightarrow t $ paths. Then 
 there is  either an independent system of edge-disjoint $ s \rightarrow t $ paths $ \mathcal{P}' $ with $ 
\mathsf{span}_{\mathcal{M}_t}(A_{last}(\mathcal{P}))\subsetneq \mathsf{span}_{\mathcal{M}_t}(A_{last}(\mathcal{P}')) $ or there is a 
$ t-s $ cut $ X $  such that the complementarity conditions (Condition \ref{Meng complementary}) hold for $ \mathcal{P} $ and $ X $. 
\end{lem}

\begin{sbiz}
 Call $ W $ an augmenting walk  if
 
 \begin{enumerate}
 \item  $ W $ is a directed walk with respect to the digraph that we obtain from $ D $ by changing the direction of edges in $ 
 A(\mathcal{P}) 
 $,
 \item $  \mathsf{start}(W)=s $ and $ W $ meets no more  $ s $,
 \item $ A(W)\triangle A(\mathcal{P}) $ is independent,
 \item if for some initial segment $ W' $ of $ W $ the set $ A(W')\triangle A(\mathcal{P}) $ is not independent, then for the one edge 
 longer initial segment $ \mathcal{W}''=\mathcal{W}'e $ the set $ A(W'')\triangle A(\mathcal{P}) $ is independent again.
 \end{enumerate}

\noindent %The redirected edges of $ W $ (i.e. edges from $ A(\mathcal{P}) $) are called backward edges, and the others are forward 
%edges. 
If there is an augmenting walk 
terminating in $ 
t $, then let $ W $ be a shortest such a walk. Build $ \mathcal{P}' $ from the edges $ A(W)\triangle A(\mathcal{P}) $ in the 
following way. 
Keep untouched those $ 
P\in \mathcal{P} $ for which $ A(W)\cap A(P)=\varnothing $ and replace the remaining finitely many paths, say $ \mathcal{Q}\subseteq 
\mathcal{P} $ where
$\left|\mathcal{Q}\right|=k $,  by $ k+1 $ new $ s \rightarrow t $ paths constructed from the edges $ A(W)\triangle A(\mathcal{Q}) $ by 
the greedy method. Obviously $ \mathcal{P}' $ is an independent system of edge-disjoint $ s\rightarrow t $ paths. We need to show that  
\[ \mathsf{span}_{\mathcal{M}_t}(A_{last}(\mathcal{P}))\subsetneq \mathsf{span}_{\mathcal{M}_t}(A_{last}(\mathcal{P}')). \] If only the 
last 
vertex of 
$ W $ is $ t $, then it is clear. Let $ f_1, e_1,\dots,f_n, e_n, f_{n+1}  $ be the edges of $ W $ incident with $ t $   enumerated with 
respect to the direction of $ W $.  
The initial segments of $ W $ up to the inner appearances of $ t $ may not be augmenting walks (since $ W $ is a shortest that terminates 
in $ t $) hence by condition 4 the one edge longer 
and the one edge shorter segments are.  It follows that for any $1\leq i \leq n $ there is a $ 
\mathcal{M}_t $-circuit $ C_i $ in  \[  A_i:= A(\mathcal{P})\cap\mathsf{in}_D(t)+f_1-e_1+f_2-e_2+\dots +f_i. \] Furthermore, $ f_i\notin 
A(\mathcal{P}) $ and $ e_i\in C_i\cap A(\mathcal{P}) $. It implies by induction that $ A_i \setminus \{ e_i \} $ spans the same set in $ 
\mathcal{M}_t $ 
as $ A(\mathcal{P})\cap\mathsf{in}_D(t) $ whenever $ 1\leq i \leq n $ and hence $ A_n\cup\{ f_{n+1} \} $  spans a strictly larger.\\

Suppose now that none of the augmenting walks terminates in $ t $. Let us denote the set of the last vertices of the  
augmenting walks by $ Y $. We show that $ \mathcal{P} $ and $ X:=V\setminus Y $ satisfy the complementarity conditions.  Obviously $ X $ 
is a $ t-s $ cut. Suppose, to the contrary, that $ e\in A(\mathcal{P})\cap \mathsf{out}_D(X) $. Pick an 
augmenting walk $ W $ terminating in $ \mathsf{head}(e) $.  
Necessarily $ e\in A(W) $, otherwise $ We $ would be an  augmenting 
walk contradicting to the definition of $ X $. Consider the initial segment $ W' $ of $ W $ for which the following edge is 
$ e $. 
Then $ W'e $ is an augmenting walk (if $ W' $ itself is not, then it is because of condition 4) which leads to the same contradiction. 

To show the second complementarity condition assume  that $ f\in \mathsf{in}_D(X)\setminus  A(\mathcal{P}) $. Choose an augmenting walk $ 
W $ that terminates in $ \mathsf{tail}(f) 
$. We may 
suppose that $ f\notin A(W) $ otherwise we consider the initial segment $ W' $ of $ W $ for which the following edge is $ f $ (it is an 
augmenting walk, otherwise  $ W'f $ would be by applying condition 4). The initial segments of $ Wf $ that terminate in $ 
\mathsf{head}(f) $ may not be augmenting walks. Let $ f_1, e_1,\dots,f_n, e_n $ be the ingoing-outgoing edge pairs of $ \mathsf{head}(f) 
$ in $ W $ with respect to the direction of $ W $ (enumerating with respect to the direction of $ W $) and let $ f_{n+1}:=f $. Then for 
any $1\leq i \leq n+1 $ there is a unique $ 
\mathcal{M} $-circuit $ C_i $ in  \[   A(\mathcal{P})\cap\mathsf{in}_D(\mathsf{head}(f))+f_1-e_1+f_2-e_2+\dots +f_i. 
\] 
It follows 
by using condition 4 and the definition of $ X $ that for $ 1\leq i \leq n $

\begin{enumerate}
\item $ f_i\notin A(\mathcal{P}) $ and $ e_i\in 
C_i\cap A(\mathcal{P}) $,
\item $ \mathsf{tail}(e_i),\mathsf{tail}(f_i)\in Y $ (tail with respect to the original direction),
\item $ C_i\subseteq \mathsf{in}_D(X) $.

\end{enumerate}
 Assume that we already know for some $ 1\leq i\leq n $ that $ f_j $ is 
spanned by $ F:= A(\mathcal{P})\cap \mathsf{in}_D(X) $ in $ \mathcal{M} $ whenever 
$ j<i $. Any 
element 
of $ C_i\setminus \{ f_i \} $ which is not in $ F $ has a form $ f_j 
$ for some $ j<i $ thus by the induction hypothesis it is spanned by $ F $ and hence we obtain that $ f_i\in 
\mathsf{span}_{\mathcal{M}}(F) $ as well. By induction it 
is true for $ i=n+1 $.
\end{sbiz}

\begin{prop}\label{nagyobb wave csinál}
Assume that $ (\mathcal{W},X) $ and $ (\mathcal{Q},Y) $ are waves where $ Y \subseteq X $ and $ \mathcal{Q} $ consists of the 
forward-continuation  of some of the paths in $ \mathcal{W} $ where the new terminal segments lie in $ X $. Let $ \mathcal{W}_Y:=\{ P\in 
\mathcal{W}: \mathsf{end}(P)\in 
Y 
\} $. Then for an appropriate $ \mathcal{Q}'\subseteq \mathcal{Q} $ the pair $ (\mathcal{W}_Y\cup \mathcal{Q}',Y) $ is a wave with $ 
(\mathcal{W},X)\leq (\mathcal{W}_Y\cup \mathcal{Q}',Y) $. 
\end{prop}
\begin{sbiz}
 The path-system $ \mathcal{W}_Y\cup \mathcal{Q} $ (not necessarily  disjoint union)  is edge-disjoint since the edges in $ 
 A(\mathcal{Q})\setminus A(\mathcal{W}) $ lie in $ X $. For the same reason it may violate independence only at the vertices $ \{ 
 \mathsf{end}(P): P\in \mathcal{W}_Y \}\subseteq Y $. Pick a base $ B $ of $\mathsf{in}_D(Y) $ for which 
  \[ A_{last}(\mathcal{W}_Y)\subseteq B \subseteq A_{last}(\mathcal{W}_Y)\cup A_{last}(\mathcal{Q}). \]
  It is routine to check that the choice  $ \mathcal{Q}'=\{ P\in \mathcal{Q}: A(P)\cap B\neq \varnothing \} $ is suitable. 
\end{sbiz}

For $ A_0\subseteq A $ let us denote $ (D-\mathsf{span}_{\mathcal{M}}(A_0),\mathcal{M}/\mathsf{span}_{\mathcal{M}}(A_0)) $ by 
$ \boldsymbol{\mathfrak{D}(A_0) }$. Note that for any $ A_0 $ the matroid corresponding to $ \mathfrak{D}(A_0)  $ has no loops and $ 
(D,\mathcal{M})=\mathfrak{D}(\varnothing)=:\boldsymbol{\mathfrak{D}} $.
\begin{obs}
If $ (\mathcal{W},X) $ is a wave and for some $ A_0\subseteq A \setminus A(\mathcal{W}) $ 
the set $ A_0\cup A(\mathcal{W}) $ 
is independent, then $ (\mathcal{W},X) $ is a $ \mathfrak{D}(A_0)  $-wave as well.
\end{obs}

\begin{lem}\label{egy él kiszed complete extension}
If $ (\mathcal{W},X) $ is a maximal $ \mathfrak{D} $-wave   and 
$ e\in A \setminus A(\mathcal{W}) $ for which $ A(\mathcal{W})\cup \{ e \} $ is independent,  then all the extensions 
 of 
the $ \mathfrak{D}(e) $-wave $ (\mathcal{W},X) $ in $ \mathfrak{D}(e) $   are complete.
\end{lem}

\begin{sbiz}
Seeking a contradiction, assume that we have an  incomplete extension $ (\mathcal{Q},Y) $ of $ (\mathcal{W},X) $ with respect to $ 
\mathfrak{D}(e) $. Observe 
that necessarily $ e\in \mathsf{in}_D(Y) $ and $ 
r_{\mathcal{M}}(\mathsf{in}_D(Y)/A_{last}(\mathcal{Q}))=1 $. Furthermore, $ Y \subsetneq X $ by Observation \ref{incomplete extension}. 
\iffalse
Since the extension 
is incomplete we also know that there is some $ P\in \mathcal{W} $ such 
that $ \mathcal{Q} $ does not contain any forward-continuation of $ P $. 
By definition of $ \leq $ it follows that $ P $ does not meet with $ Y $ hence $ \mathsf{end}(P) $ ensures $ Y \subsetneq X $.
\fi 

 We show that $ (\mathcal{W},X) $ has a proper extension with respect to 
 $ \mathfrak{D} $ as well contradicting to its 
maximality.  Without loss of generality we may assume that $ \mathsf{in}_D(X)=A_{last}(\mathcal{W}) $. Indeed, otherwise 
we  delete the edges $ \mathsf{in}_D(X)\setminus A(\mathcal{W}) $ from $ D $ and from $ \mathcal{M} $. It is 
routine to check that after the deletion $ (\mathcal{W},X) $ is still a  wave and a proper extension of it remains a proper extension 
after putting back these edges. 

 Contract $ V \setminus X $ to $ s $  and 
contract $ Y $ to $ t $ in $ D $ and keep $ \mathcal{M} $ unchanged. Apply the 
augmenting 
walk method (Lemma \ref{augmenting walk basic}) in the resulting system with the   $ V \setminus X\rightarrow Y $ terminal segments 
of the paths in $ \mathcal{Q} $. 
If the 
augmentation is 
possible, then the 
assumption $ 
\mathsf{in}_D(X)=A_{last}(\mathcal{W}) $ ensures  
that the first edge of any element of the resulting path-system $ \mathcal{R} $ is a last edge of some path in $ \mathcal{W} $. By 
uniting the elements of $ \mathcal{R} $ with the corresponding paths from $ \mathcal{W} $ 
 we can get a new independent system of edge-disjoint  $ s\rightarrow Y $ paths $ 
\mathcal{Q}' $ (with 
respect to $ \mathfrak{D} $). Furthermore, $ r_{\mathcal{M}}(\mathsf{in}_D(Y)/A_{last}(\mathcal{Q}))=1 $ 
guarantees that $ A_{last}(\mathcal{Q}') $ spans $ \mathsf{in}_D(Y) $ in $ \mathcal{M} $ and hence  $ (\mathcal{Q}',Y) $ is a wave. Thus 
by 
Proposition \ref{nagyobb wave csinál} we get an extension of $ (\mathcal{W},X) $ and it is proper because $ Y\subsetneq X $ which is 
impossible.

Thus the augmentation must be unsuccessful which implies by Lemma \ref{augmenting walk basic} that there is some $ Z $ with $ Y \subseteq 
Z \subseteq X $  such that  $ Z 
$ and 
$ \mathcal{Q} $ satisfy the complementarity conditions. By Proposition \ref{incomplete extension} we know 
that $ Z \subsetneq X$. For the
initial segments $ \mathcal{Q}_Z $ of the paths in $ \mathcal{Q} $ 
up to $ Z $  the pair $ (\mathcal{Q}_Z,Z) $ forms a 
 wave.  Thus  by applying Proposition \ref{nagyobb wave csinál}  with $ (\mathcal{W},X) $ and $  (\mathcal{Q}_Z,Z) $ we obtain an 
 extension 
 of $ (\mathcal{W},X) $ which is proper because $ Z \subsetneq X $ contradicting to the maximality of $ (\mathcal{W},X) $.
\end{sbiz}

\begin{prop}\label{t elér max wave}
If $ (\mathcal{W},X) $ is a maximal wave and $ v\in X $, then there is a $ 
v\rightarrow t $ path $ Q $ in $ D[X] $  such that  $ A(\mathcal{W})\cup A(Q) $ is independent.
\end{prop}

\begin{sbiz}
 It is 
equivalent to show that there exists a
$ v\rightarrow t $ path $ Q $ in  $ D-\mathsf{span}_{\mathcal{M}}(A(W)) $ (path $ Q $ will  necessarily lie in $ D[X] $ because 
$ 
D-\mathsf{span}_{\mathcal{M}}(A(W)) $ does not contain any edge entering into $ X $.) Suppose, to the contrary, that it is not the case. 
Let $ X' 
\subsetneq X $ be the set 
of 
those 
vertices in $ X $  
that are unreachable from $ v $ in $ D-\mathsf{span}_{\mathcal{M}}(A(W)) $ (note that $ v\notin X' $ but $ t\in X' $ by the indirect 
assumption). Let $ 
\mathcal{W}' $ be consist of 
the paths in $ \mathcal{W} $ that meet $ X' $. 
If we prove that $ (\mathcal{W}',X') $ is a wave, 
then we are done since it would be a proper extension of the maximal wave $ (\mathcal{W},X) $. Assume that $ f\in 
\mathsf{in}_D(X')\setminus A(\mathcal{W}') $. Then by the definition of $ X' 
$ we have $ \mathsf{tail}(f)\in V \setminus X $ thus $ f\in \mathsf{in}_D(X) $. Hence $ f $ is spanned by the last edges of the paths in 
$ 
\mathcal{W} $ terminating in $ \mathsf{head}(f) $ and all these paths are in $ \mathcal{W}' $ as well. Therefore $ (\mathcal{W}',X') $ is 
a wave. 
\end{sbiz}

\begin{lem}\label{egy út végigvisz}
Let $ (\mathcal{W},X_0) $ be a maximal wave  and assume that $ P\in \mathcal{W} $ and let $ 
\mathcal{W}_0=\mathcal{W}\setminus \{ P \} $. Then there is an 
$ s $-arborescence $ \mathcal{A} $   
  such that
\begin{enumerate}
\item $ A(P)\subseteq A(\mathcal{A}) $,
\item $ A(\mathcal{A})\cap A(\mathcal{W}_0)=\varnothing $,
\item $ A(\mathcal{A})\cup A(\mathcal{W}_0) $ is independent,
\item $ t\in V(\mathcal{A}) $,
\item there is a maximal wave with respect to $ \mathfrak{D}( A(\mathcal{A})) $ which is a complete extension 
of the $ \mathfrak{D}(A(\mathcal{A})) $-wave $ (\mathcal{W}_0, X_0) $.
\end{enumerate}   
\end{lem}

\begin{ssbiz}

\begin{prop}
The pair $ (\mathcal{W}_0,X_0)= (\mathcal{W}\setminus \{ P \}, X_0) $ is a maximal wave with respect to $ \mathfrak{D}(A(P)) $.
\end{prop}
\begin{sbiz}
It is clearly a wave thus we show just the maximality. Suppose that $ (\mathcal{Q},Y) $ is a proper extension of $ (\mathcal{W}\setminus 
\{ P \}, X_0) $ with respect to $ \mathfrak{D}(A(P)) $. Necessarily $ \mathsf{end}(P)\in Y $ otherwise it would be a wave with 
respect to $ \mathfrak{D} $ which properly extends $ (\mathcal{W},X_0) $. Let $ e $ be the last edge of $ P $. We know that $ 
A_{last}(\mathcal{Q}) $ spans $ 
\mathsf{in}_D(Y) $ in $ \mathcal{M}/e $. Since $ A(\mathcal{Q}) $ is $ [\mathcal{M}/\mathsf{span}_{\mathcal{M}}(A(P))] $-independent it 
follows that $ (\mathcal{Q}\cup 
\{ P \},Y) $ is a $ \mathfrak{D} $-wave. But then it properly extends $ (\mathcal{W},X_0) $ which is a contradiction.
\end{sbiz}\\

Fix a well-ordering of $ A $ with order type $ \left|A\right|\leq \omega $.  We 
 build the arborescence $ \mathcal{A} $ by recursion. Let $ \mathcal{A}_0:=P $.  Assume that $ \mathcal{A}_m, 
\mathcal{W}_m $ and $ X_m $ 
has already defined  for $ m\leq n $ in such a way that 

\begin{enumerate}
\item $ A(\mathcal{A}_m)\cap A(\mathcal{W}_m)=\varnothing $,
\item $ A(\mathcal{A}_m)\cup A(\mathcal{W}_m) $ is independent,
\item $ (\mathcal{W}_m,X_m) $ is a maximal wave with respect to $\mathfrak{D}_m:= \mathfrak{D}(A(\mathcal{A}_m)) $ and a 
complete extension of the $ \mathfrak{D}_m $-wave $ (\mathcal{W}_k,X_k) $  whenever $ k<m $,
\item for $ 0 \leq k<n $ we have $ \mathcal{A}_{k+1}=\mathcal{A}_k+e_k $ for some $e_k\in  \mathsf{out}_{D}(V(\mathcal{A}_k)) $.
\end{enumerate}

\noindent If $ t\in V(\mathcal{A}_n) $, then $ \mathcal{A}_n $ satisfies the requirements of Lemma \ref{egy út végigvisz} thus we are 
done. Hence we may 
assume that $ t\notin V(\mathcal{A}_n) $.

\begin{prop}
$ \mathsf{out}_{D-\mathsf{span}_{\mathcal{M}}(A(\mathcal{W}_n))}(V(\mathcal{A}_n)) \neq \varnothing $.
\end{prop}
\begin{sbiz}
We claim that the  $ \mathfrak{D}_n $-wave  $ (\mathcal{W}_n,X_n) $ is  not a  $ \mathfrak{D} $-wave. Indeed, suppose it is, 
then $ 
\mathsf{end}(P)\notin X_n $ (since $ 
  A_{last}(\mathcal{W}_n) $ does not span the last edge $ e $ of 
  $ P $)  and therefore $ X_n \subsetneq X_0 $ thus it extends $ (\mathcal{W},X_0) $ properly with respect to $ \mathfrak{D} $ 
  contradicting to the 
  maximality of $ (\mathcal{W},X_0) $.   Hence the $ s 
$-arborescence $ \mathcal{A}_n $ need to have 
an edge $ e\in \mathsf{in}_D(X) $. Let $ Q $ be a path that we obtain by applying Proposition 
\ref{t elér max wave} with $ (\mathcal{W}_n,X_n) $ and $ \mathsf{head}(e) $ in the system $ \mathfrak{D}_n $. Consider the last vertex 
$ v $ of $ Q $ which is in $ 
V(\mathcal{A}_n) $. Since $ v\neq t $  there is an outgoing edge $ f $ of $ v $ in $ Q $ and hence $ f\in 
\mathsf{out}_{D-\mathsf{span}_{\mathcal{M}}(A(\mathcal{W}_n))}(V(\mathcal{A}_n)) $. 
\end{sbiz}\\

  Pick the smallest element $ e_n $ of $ \mathsf{out}_{D-\mathsf{span}_{\mathcal{M}}(A(\mathcal{W}_n))}(V(\mathcal{A}_n))  $ and let $ 
  \mathcal{A}_{n+1}:=\mathcal{A}_n+e_n $.  
  Let $ (\mathcal{W}_{n+1},X_{n+1}) $ be a maximal 
  wave 
  with respect to $ \mathfrak{D}_{n+1}$ which extends $ (\mathcal{W}_n,X_n) $ (exists 
  by Corollary \ref{maximal wave} ). Lemma 
  \ref{egy él kiszed complete extension} ensures that it is a complete extension.
  
  Suppose, to the contrary, that the recursion does not stop after finitely many steps. Let  
  \[ \mathcal{A}_{\infty}:=\left( \bigcup_{n=0}^{\infty}V(\mathcal{A}_n), \bigcup_{n=0}^{\infty}A(\mathcal{A}_n)  \right). \] Note that $ 
  A(\mathcal{A}_\infty) $ is independent and $ \left\langle 
  (\mathcal{W}_n,X_n): n<\omega \right\rangle $ is an $ \leq 
  $-increasing sequence of $ \mathfrak{D}(A(\mathcal{A}_{\infty})) $-waves. Let $ (\mathcal{W}_\infty, 
  X_\infty) $ be a maximal $ \mathfrak{D}(A(\mathcal{A}_{\infty})) $-wave  which extends $ \sup_n (\mathcal{W}_n,X_n) $ 
  (see Lemma \ref{sup of waves}). 
  
  It may not be a wave with respect to $ \mathfrak{D} $.  Hence the $ s $-arborescence $ 
  \mathcal{A}_{\infty} $ contains an edge $ e\in 
  \mathsf{in}_D(X_{\infty}) $. Apply Proposition 
  \ref{t elér max wave} with $ (\mathcal{W}_\infty,X_\infty) $ and $ \mathsf{head}(e) $ in the system $ 
  \mathfrak{D}(A(\mathcal{A}_\infty)) 
  $.  Consider the last vertex $ v $ of the resulting $ Q $ which is in $ 
  V(\mathcal{A}_\infty) $. Since $ v\neq t $ by assumption there is an outgoing edge $ f $ of $ v $ in $ Q $. Then $ f\in 
  \mathsf{out}_{D-\mathsf{span}_{\mathcal{M}}(A(\mathcal{W}_\infty))}(V(\mathcal{A}_\infty)) $ which implies that for some $ n_0<\omega $ 
  we have $ f\in 
    \mathsf{out}_{D-\mathsf{span}_{\mathcal{M}}(A(\mathcal{W}_n))}(V(\mathcal{A}_n)) $ whenever $n>n_0 $. But then the infinitely many 
    pairwise distinct edges $ \{ 
    e_n: n_0<n<\omega \} $ are all smaller than $ f $ in our fixed well-ordering on $ A $ which contradicts to the fact that the type of  
    this well-ordering is at most $ \omega $.  
\end{ssbiz}\\

The Theorem follows easily from  Lemma \ref{egy út végigvisz}. Indeed, pick a maximal wave $ (\mathcal{W}_0,X_0) $ with respect to $ 
\mathfrak{D}_0:= \mathfrak{D} $ where $ \mathcal{W}_0=\{ P_n 
\}_{n<\omega} $. Apply Lemma \ref{egy út végigvisz} with $ P_0\in \mathcal{W}_0 
$. The resulting arborescence $ \mathcal{A}_0 $  
contain a unique $ s \rightarrow t $ path $ P_0^{*} $ which is necessarily a 
forward-continuation of $ P_0 $ (usage of a new edge from $ \mathsf{in}_D(X_0) $ would lead to dependence).  Then by Lemma \ref{egy út 
végigvisz} we have a 
maximal wave $ (\mathcal{W}_1,X_1) $ (where $ X_1 \subseteq X_0 $) with respect to $ \mathfrak{D}_1:=\mathfrak{D}_0(A(\mathcal{A}_0)) $ 
such that $ 
\mathcal{W}_1=\{ P_n^{1} \}_{1 \leq n<\omega} 
$ where $ P_n^{1} $ is a 
forward continuation of $ P_n $. Then we apply Lemma \ref{egy út végigvisz} with the $ \mathfrak{D}_1 $-wave $ (\mathcal{W}_1,X_1) $  and 
$ P_1^{1}\in \mathcal{W}_1 $  and continue the process recursively. 
By the construction  $ \bigcup_{n< m}A(P_n^{*}) $ is independent for each $ m<\omega $. Since 
 $ \mathcal{M} $ is finitary $ \bigcup_{n< \infty}A(P_n^{*}) $ is independent as well thus 
 $ \mathcal{P}:=\{ P_n^{*} \}_{n<\omega} $ is a desired paths-system that satisfies the complementarity conditions with $ X_0 $.
\end{proof}

\section{Open problems}

We suspect that one can omit the countability condition for $ D $ in Theorem \ref{főtétel  atroid mengerb}  by analysing the famous 
infinite Menger's theorem \cite{aharoni2009menger} of Aharoni and Berger. We also think that it is possible to put matroid constraints on 
the outgoing edges of each vertex as well but this generalization contains  the Matroid intersection conjecture for finitary matroids as 
a special case, which problem is hard enough itself. 
The finitarity of the matroids  are used several times in the proof; we do not know yet if one can omit this condition.

\end{document}